\documentclass[a4paper,12pt]{article}
\usepackage{amscd,amsfonts,amsmath,amsthm,amssymb,bm,color,latexsym}

\newtheorem{conj}{Conjecture}[section]

\newtheorem{prop}{Proposition}[section]

\newcommand{\F}{\mathcal{F}}
\newcommand{\J}{\mathcal{J}}
\newcommand{\I}{\mathcal{I}}
\newcommand{\A}{\mathcal{A}}
\newcommand{\Q}{\mathbb{Q}}
\newcommand{\N}{\mathbb{N}}

\begin{document}

\title{Remarks on Frankl's conjecture}
\author{Francesco Marigo\thanks{University of Insubria}\ \ and Davide Schipani\thanks{University of Zurich}}
\date{}
\maketitle
\abstract{First a few reformulations of Frankl's conjecture are given, in terms of reduced families or matrices, or analogously in terms of lattices. These lead naturally to a stronger conjecture with a neat formulation which might be easier to attack than Frankl's. To this end we prove an inequality which might help in proving the stronger conjecture.}

\section{Introduction}

Frankl's conjecture, also called union-closed sets conjecture, is a well known open problem in combinatorics since 1979. The simple and elementary formulation can certainly account in part for the appeal of the problem. Briefly, a family of sets is said to be union-closed if the union of any two sets of the family is still part of the family; the conjecture states that for any finite union-closed family of finite sets, other than the family consisting only of the empty set, there exists an element that belongs to at least half of the sets in the family. A dual intersection-closed sets formulation exists as well. A survey of the state of the art can be found in \cite{bs}.

The aim of this paper is to present new reformulations, which lead naturally to a stronger conjecture. The strucure of the paper is as follows. Section \ref{notation} fixes the notation and definitions which will be used throughout the paper. Section \ref{equival} presents a few reformulations of the conjecture in terms of reduced families, matrices and lattices. In particular the matrix setting will be mostly exploited, although, as shown in Section \ref{notation}, there exists a correspondence among all settings. Section \ref{stronger} presents a stronger conjecture, which naturally derives from a reformulation of the previous section. It has a neat formulation in terms of cardinality of intersections or in terms of matrices. Section \ref{ineq} proves a combinatorial inequality strictly related to the stronger conjecture.

\section{Notation and definitions}\label{notation}

In this section we set the terminology used in the paper.
First of all, we remark that all sets in this paper are finite.
Terms not defined here and statements without proof are either covered in the literature, for example \cite{bf,bs,gg}, or are a direct consequence of definitions, possibly after a bit of elaboration. 

Let $\F$ be a family of subsets of a nonempty set $S$ such that:
\begin{itemize}
\item $\varnothing$ and $S$ belong to $\F$;
\item for every $x \in S$, there are $A,B \in \F$ such that $A \setminus B = \{ x \}$.
\end{itemize}
We call $\F$ a \emph{reduced family} on $S$. If it holds
\begin{itemize}
\item for all $A,B \in \F$, also $A \cap B$ belongs to $\F$,
\end{itemize}
we say that $\F$ is a reduced \emph{intersection-closed} family. If it holds
\begin{itemize}
\item for all $A,B \in \F$, also $A \cup B$ belongs to $\F$,
\end{itemize}
we say that $\F$ is a reduced \emph{union-closed} family.
The family of complements of sets in a reduced intersection-closed family is a reduced union-closed family and vice-versa.

There is a direct correspondence between reduced families and lattices.
Indeed, the members of a reduced intersection-closed (union-closed) family $\F$ form a lattice with respect to set inclusion.
The other direction can be shown as follows.
Let $L$ be a lattice (we always assume that a lattice have at least two elements). A \emph{filter} of $L$ is an upward closed and meet-closed subset of $L$, an \emph{ideal} of $L$ is a downward closed and join-closed subset of $L$. We say that a filter (ideal) $A$ is \emph{irreducible} if it is not the intersection of any set of filters (ideals) which does not have $A$ as member. We write $\J$ for the family of irreducible filters of $L$ and $\I$ for the set of irreducible ideals of $L$.
For $a \in L$, we write $\J_a$ for the subset of $\J$ of all the irreducible filters which contain $a$, and $\I_a$ for the subset of $\I$ of all the irreducible ideals which do not contain $a$. Then

\begin{itemize}
\item the family of all $\J_a$'s, for $a \in L$, is a reduced intersection-closed family on $\J$;
\item the family of all $\I_a$'s, for $a \in L$, is a reduced union-closed family on $\I$.
\end{itemize}
And the lattice of the members of either of the two families, ordered by set inclusion, is isomorphic to $L$.

Now, let $\F = \{ A_1, \ldots , A_m \}$ be a nonempty family of subsets of a nonempty set $S = \{ x_1, \ldots , x_n \}$.
We call \emph{characteristic matrix} of $\F$ the matrix $ F \in \Q^{m \times n}$ with elements
$$F(h,k) = \left\{
\begin{array}{ll}
1 &\hbox{ if } x_k \in A_h\\
0 &\hbox{ if } x_k \notin A_h.
\end{array}
\right.
$$
A matrix $F \in \Q^{m \times n}$ is the characteristic matrix of some reduced family iff:
\begin{itemize}
\item there is $h \in \{1,\ldots,m\}$ such that $F(h,k)=0$ for all $k \in \{1,\ldots,n\}$;
\item there is $h \in \{1,\ldots,m\}$ such that $F(h,k)=1$ for all $k \in \{1,\ldots,n\}$;
\item the rows are all different;
\item for all $k \in \{1,\ldots,n\}$ there are two rows which differ only at the $k$-th column (hence the columns are all different).
\end{itemize}
The characteristic matrix of a reduced family is the characteristic matrix of some reduced intersection-closed family iff:
\begin{itemize}
\item for all $h,i \in \{1,\ldots,m\}$, there is $j \in \{1,\ldots,m\}$ such that $F(j,k)=F(h,k)F(i,k)$
for all $k \in \{1,\ldots,n\}$;
\end{itemize}
it is the characteristic matrix of some reduced union-closed family iff:
\begin{itemize}
\item for all $h,i \in \{1,\ldots,m\}$, there is $j \in \{1,\ldots,m\}$ such that $F(j,k)=F(h,k)+F(i,k)-F(h,k)F(i,k)$
for all $k \in \{1,\ldots,n\}$.
\end{itemize}
We call \emph{reduced intersection-closed (union-closed)} matrix a characteristic matrix of some reduced intersection (union) closed family. We write $\Q_\cap^{m \times n}$ ($\Q_\cup^{m \times n}$) for the set of $m \times n$ reduced intersection-closed (union-closed) matrices.

We write $\F_x$ for the subfamily of $\F$ of all the sets which contain $x$.
As usual, $|A|$ is the cardinality of the set $A$.
We write $e_k$ for the $k$-th unit coordinate vector, that is the vector with all components $0$ except the $k$-th component, which is equal to $1$.
If there is no risk of confusion, we write shortly $\Sigma X, \Sigma x$ for the sum of all the elements of a matrix $X$ or a vector $x$, respectively.
Furthermore, the vectors written in matrix form are columns, that is $a \times 1$ matrices. 

\section{Equivalent formulations}\label{equival}

We formulate Frankl's conjecture for lattices, reduced families and reduced matrices. It is known \cite{bs} that working with reduced families is not restrictive.
It is straightforward to check that the following statements are all equivalent.
Each equivalent formulation of the conjecture is followed by a reinforced statement, which corresponds to a stronger conjecture by B. Poonen \cite{pb,bs}.

\begin{conj} \label{conj:latfil}
Let $L$ be a lattice, $\J$ the set of its irreducible filters.
Then,
$$2 \min\{|A| : A \in \J\} \le|L|.$$
\end{conj}
\emph{Reinforcement}: The inequality is strict unless $L$ is a Boolean lattice.
\begin{conj} \label{conj:latide}
Let $L$ be a lattice, $\I$ the set of its irreducible ideals.
Then,
$$2 \max\{|A^C| : A \in \I\} \ge|L|.$$
\end{conj}
\emph{Reinforcement}: The inequality is strict unless $L$ is a Boolean lattice.
\begin{conj} \label{conj:setint}
Let $\F$ be a reduced intersection-closed family on $S$.
Then,
$$2 \min\{|\F_x| : x \in S\} \le |\F|.$$
\end{conj}
\emph{Reinforcement}: The inequality is strict unless $\F$ is the powerset of $S$.
\begin{conj} \label{conj:setuni}
Let $\F$ be a reduced union-closed family on $S$.
Then,
$$2 \max\{|\F_x| : x \in S\} \ge |\F|.$$
\end{conj}
\emph{Reinforcement}: The inequality is strict unless $\F$ is the powerset of $S$.
\begin{conj} \label{conj:matint}
Let $F \in \Q^{m \times n}$ be a reduced intersection-closed matrix.
Then,
$$2 \min\{\Sigma [F e_k] : k=1,\ldots,n\} \le m.$$
\end{conj}
\emph{Reinforcement}: The inequality is strict unless $m=2^n$
\begin{conj} \label{conj:matuni}
Let $F \in \Q^{m \times n}$ be a reduced union-closed matrix.
Then,
$$2 \max\{\Sigma [F e_k] : k=1,\ldots,n\} \ge m.$$
\end{conj}
\emph{Reinforcement}: The inequality is strict unless $m=2^n$.

We refer to any of these equivalent conjectures and their reinforcements as \emph{(reinforced) Frankl's conjecture}.\\

Let $\Delta_n \subseteq \Q^n$ be the $(n-1)$-dimensional symplex generated by unit coordinate vectors $e_k$, for $k \in \{1,\ldots,n\}$, that is
$$\Delta_n = \{x \in \Q^n : x_k \ge 0 \ (k=1,\ldots,n) , \quad \Sigma x = 1\}.$$

\begin{prop} \label{prop:conjeq1}
Frankl's conjecture is equivalent to the following statement.
\end{prop}

\begin{conj} \label{conj:matcom}
Let $F \in \Q^{m \times n}$ be a reduced union-closed matrix.
Then, there is $x \in \Delta_n$ such that $2\Sigma [F x]  \ge m$.
\end{conj}

\begin{proof}(of Proposition \ref{prop:conjeq1}).
Conjecture \ref{conj:matuni} is a particular case of Conjecture \ref{conj:matcom}.
Conversely, suppose that Conjecture \ref{conj:matuni} do not hold.
Then, for every $k \in \{1,\ldots,n\}$ it is $2\Sigma [F e_k]< m$.
Therefore, for every $x \in \Delta_n$ it would hold
$$2\Sigma [F x] = 2 \sum_{i=1}^n x_i \Sigma [F e_i] \le 2 \max\{\Sigma [F e_k] : k=1,\ldots,n\}<m,$$
hence Conjecture \ref{conj:matcom} would not hold.
\end{proof}
The formulation of Conjecture \ref{conj:matcom} suggests a strategy to try to solve the problem. We illustrate this strategy by stating another conjecture, clearly equivalent to Conjecture \ref{conj:matcom}:

\begin{conj} \label{conj:matfun}
There is a weight function $w: \Q_\cup^{m \times n} \to \Delta_n$ such that,
for every reduced union-closed matrix $F \in \Q^{m \times n}$, the following holds:
$2\Sigma [F w(F)] \ge m$.
\end{conj}

The strategy, then, becomes to find a proper weight function $w$, and the rest of the paper will be in this direction.

We choose a simple sequence of functions $(w_0,w_1,w_2\ldots)$ where, if Frankl's conjecture holds, a function $w_r$ can be found which satisfies Conjecture \ref{conj:matfun}.
Let $w_r : \Q_\cup^{m \times n} \to \Delta_n$ be the function which sends a matrix $F$ to a vector whose components are proportional to the $r$-th powers of the sums of columns:
$$[w_r(F)]_k = \frac {(\Sigma[F e_k])^r}{\sum_{i=1}^n(\Sigma[F e_i])^r}.$$
The more $r$ increases, the more the weight is concentrated on the column of highest sum. Hence,

\begin{prop} \label{prop:conjeq2}
Reinforced Frankl's conjecture is equivalent to the following statement.
\end{prop}

\begin{conj} \label{conj:powers}
Let $F \in \Q^{m \times n}$ be a reduced union-closed matrix.
Then, there is $r \in \N$ such that $2\Sigma [F w_r(F)] \ge m$.
\end{conj}

\begin{proof}(of Proposition \ref{prop:conjeq2}).
We prove first that reinforced Conjecture \ref{conj:matuni} implies Conjecture \ref{conj:powers} in case $m \neq 2^n$. 
If reinforced Conjecture \ref{conj:matuni} holds, then
$$2 \lim_{t \to \infty}\Sigma [F w_t(F)] = 2 \max\{\Sigma [F e_k] : k=1,\ldots,n\} > m,$$
hence there is $r$ which satisfies Conjecture \ref{conj:powers}.
In the other direction Conjecture \ref{conj:powers} implies Conjecture \ref{conj:matcom}, therefore Conjecture \ref{conj:matuni}.
For the case $m= 2^n$, for any $t$ we have all equalities
$$2 \Sigma [F w_t(F)] = 2 \max\{\Sigma [F e_k] : k=1,\ldots,n\} = m,$$
as $\Sigma [F w_t(F)]$ is the same for every $t$, being a weighted mean of elements all equal to $m/2$.
\end{proof}

Conjecture \ref{conj:powers} has a nice interpretation in terms of union-closed families.

\begin{prop} \label{prop:conjeq3}
Conjecture \ref{conj:powers} is equivalent to the following statement.
\end{prop}

\begin{conj} \label{conj:inters}
Let $\F$ be a reduced union-closed family on a set $S$.
Then, there is $r \in \N$ such that the average cardinality of intersections of $r$-tuples of members of $\F$ is not greater than twice the average cardinality of intersections of $(r+1)$-tuples of members of $\F$, in formulas:
$$2\frac{\sum_{\A \in \F^{r+1}}|\bigcap \A|}{|\F|^{r+1}} \ge \frac{\sum_{\A \in \F^r}|\bigcap \A|}{|\F|^r}.$$
\end{conj}

\begin{proof}(of Proposition \ref{prop:conjeq3}).
Conjecture \ref{conj:powers} states, equivalently, that there is $r \in N$ such that
$$2\frac {\sum_{i=1}^n(\Sigma[F e_i])^{r+1}}{\sum_{i=1}^n(\Sigma[F e_i])^r} \ge m.$$
Translating this inequality in terms of the union-closed family we have
$$2\frac {\sum_{x \in S}|\F_x|^{r+1}}{\sum_{x \in S}|\F_x|^r} \ge |\F|,$$
which can be written as
$$2\frac{\sum_{x \in S}|\F_x|^{r+1}}{|\F|^{r+1}} \ge \frac{\sum_{x \in S}|\F_x|^r}{|\F|^r}.$$
The thesis follows from the equalities
$$\sum_{x \in S}|\F_x|^r = \sum_{x \in S}\left|\left\{\A : \A \in \F^r , \ x \in \bigcap \A \right\}\right|=$$
$$\sum_{\A \in \F^r}\left|\left\{x : x \in S , \ x \in \bigcap \A \right\}\right| = \sum_{\A \in \F^r}\left|\bigcap \A\right|.$$
\end{proof}

\section{A stronger conjecture}\label{stronger}

If $r=0$  in Conjecture \ref{conj:inters}, then the inequality becomes
$$2\frac{\sum_{A \in \F}|A|}{|\F|} \ge |S|,$$
which means that the average cardinality of members of $\F$ is at least half of the cardinality of $S$.
This is not true in general, a counterexample is the family $\F = \{\varnothing,\{a\},\{b\},\{a,b\},\{a,b,c\}\}$ on the set $S=\{a,b,c\}$.

If $r=1$ in Conjecture \ref{conj:inters}, then the inequality becomes
$$2\frac{\sum_{(A,B) \in \F^2}|A \cap B|}{|\F|^2} \ge \frac{\sum_{A \in \F}|A|}{|\F|}.$$
We have not found counterexamples to this inequality, which is clearly stronger than Frankl's conjecture.
We do not have counterexamples to the analogous inequalities for $r>1$ either,
but from now on we restrict our attention to the case $r=1$, which admits a particularly simple formulation.
Of course, this case would become of little interest regarding Frankl's conjecture, once a counterexample be found.
We refer to the inequality for $r=1$ and equivalent ones as \emph{strong union-closed sets conjecture}:

\begin{conj} \label{conj:inttwo}
Let $\F$ be a reduced union-closed family on a set $S$.
Then, the average cardinality of members of $\F$ is not greater than twice the average cardinality of intersections of ordered pairs of members of $\F$.
\end{conj}

The inequality of Conjecture \ref{conj:inttwo} can be written as
$$2\sum_{(A,B) \in \F^2}|A \cap B| \ge \sum_{(A,B) \in \F^2}|A|,$$
$$\sum_{(A,B) \in \F^2}|A \cap B| \ge \sum_{(A,B) \in \F^2}\left(|A| - |A \cap B|\right),$$
$$\sum_{(A,B) \in \F^2}|A \cap B| \ge \sum_{(A,B) \in \F^2}|A \setminus B|.$$
Hence we have the reformulation:

\begin{conj} \label{conj:intdif}
Let $\F$ be a reduced union-closed family on a set $S$.
Then, the average cardinality of differences of ordered pairs of members of $\F$ is not greater than the average cardinality of intersections of ordered pairs of members of $\F$.
\end{conj}

Adding the inequality of Conjecture \ref{conj:intdif} with the analog one with the role of $A$ and $B$ switched, and the equality of the sum of intersections with itself, we obtain
$$3\sum_{(A,B) \in \F^2}|A \cap B| \ge \sum_{(A,B) \in \F^2}|A \cup B|.$$
Hence we have the reformulation:

\begin{conj} \label{conj:intuni}
Let $\F$ be a reduced union-closed family on a set $S$.
Then, the average cardinality of unions of ordered pairs of members of $\F$ is not greater than three times the average cardinality of intersections of ordered pairs of members of $\F$.
\end{conj}

We translate the strong union-closed conjecture in terms of union-closed matrices.
Let $\F = \{ A_1, \ldots , A_m \}$ be a nonempty family of subsets of a nonempty set $S = \{ x_1, \ldots , x_n \}$, and $F$ its characteristic matrix. We write $F^T$ for the transpose of $F$ and $\overline{F}$ for the characteristic matrix of the family of complements of members of $\F$,
that is $\overline{F}(h,k)=1-F(h,k)$. The following equalities are straightforward:
$$[FF^T](h,k) = |A_h \cap A_k|,$$
$$[F\overline{F}^T](h,k) = |A_h \setminus A_k|,$$
$$[\overline{F}F^T](h,k) = |A_k \setminus A_h|,$$
$$[\overline{F}\overline{F}^T](h,k) = |A_h^C \cap A_k^C|.$$
Then, Conjecture \ref{conj:intdif} becomes:

\begin{conj} \label{conj:matdif}
Let $F$ be a reduced union-closed matrix.
Then, it holds
$$\Sigma[FF^T] \ge \Sigma[F\overline{F}^T],$$
$$\Sigma[FF^T] \ge \Sigma[\overline{F}F^T].$$
\end{conj}

\section{A useful inequality}\label{ineq}

In support of Conjecture \ref{conj:intdif}, we prove an inequality which holds for any nonempty family $\F$ of subsets of a nonempty set $S$.
For every $A,B \in \F$, let $A \oplus B$ be the symmetric difference, corresponding to the Boolean $\textsf{XOR}$ operation and $A \otimes B$ be the complement of the symmetric difference, corresponding to the Boolean $\textsf{XNOR}$ operation. Basic equalities are:\\
$A \oplus B = (A \cup B) \cap (A^C \cup B^C)$, $A \oplus B = (A \setminus B) \cup (B \setminus A)$,\\
$A \otimes B = (A \cap B) \cup (A^C \cap B^C)$, $A \otimes B = (A \to B) \cap (B \to A)$,\\
$A \otimes B = (A \oplus B)^C$.

\begin{prop} \label{prop:symdif}
Let $\F$ be a nonempty family of subsets of a nonempty set $S$.
Then, the average cardinality of symmetric differences of ordered pairs of members of $\F$ is not greater than the average cardinality of complements of symmetric differences of ordered pairs of members of $\F$. In formulas,
$$\sum_{(A,B) \in \F^2}|A \otimes B| \ge \sum_{(A,B) \in \F^2}|A \oplus B|.$$
The inequality is strict unless every $x \in S$ belongs to exactly half of the members of $\F$.
\end{prop}

Proposition \ref{prop:symdif} follows from the next proposition. for every $x \in S$, let
$\delta_x = |\F_x|-|\F_x^C| = |\{A \in \F : x \in A \}| - |\{A \in \F : x \notin A \}|$.

\begin{prop} \label{prop:symdif}
Let $\F$ be a nonempty family of subsets of a nonempty set $S$.
Then,
$$\sum_{(A,B) \in \F^2}|A \otimes B| = \sum_{(A,B) \in \F^2}|A \oplus B| + \sum_{x \in S}\delta_x^2.$$
\end{prop}

\begin{proof}
Let $A,B \in \Q^{m \times n}$ be two matrices. It holds
$$(A-B)(A-B)^T = AA^T - AB^T - BA^T + BB^T,$$
hence
$$\Sigma[AA^T]+\Sigma[BB^T]=\Sigma[AB^T]+\Sigma[BA^T]+\Sigma[(A-B)(A-B)^T].$$
For a matrix $C \in \Q^{m \times n}$, the number $\Sigma CC^T$ is the scalar product of the vector, sum of all the rows of the matrix, with itself, indeed
$$\Sigma CC^T = \sum_{i=1}^m \sum_{j=1}^m e_i^T C C^T e_j = \left[ \sum_{h=1}^m e_h^T C \right] \left[ \sum_{h=1}^m e_h^T C \right]^T,$$
hence it is
$$\Sigma CC^T = \sum_{k=1}^n (\Sigma[C e_k])^2.$$
Then, for $F$ characteristic matrix of $\F$, putting together the previous equalities with $A=F$, $B=\overline{F}$ and $C=A-B$
we have
$$\Sigma[FF^T]+\Sigma[\overline{F}\overline{F}^T]=\Sigma[F\overline{F}^T]+\Sigma[\overline{F}F^T]+\sum_{k=1}^n \left(\Sigma[(F-\overline{F}) e_k]\right)^2.$$
This equality, translated in terms of the family of sets, becomes
$$\sum_{(A,B) \in \F^2}(|A \cap B| +|A^C \cap B^C|) = \sum_{(A,B) \in \F^2}(|A \setminus B| + |B \setminus A| )+ \sum_{x \in S}\delta_x^2,$$
whence the thesis follows.
\end{proof}

From the last equality it follows directly: 
\begin{prop} \label{prop:symdifcor}
Let $\F$ be a reduced union-closed family on a set $S$.
If $$\sum_{(A,B) \in \F^2}|A \cap B| + \sum_{x \in S}\delta_x^2 \ge \sum_{(A,B) \in \F^2}|A^C \cap B^C|,$$
then Conjecture \ref{conj:intdif} is true.
\end{prop}


\begin{thebibliography}{99}

\bibitem{bf} G. Birkhoff, O. Frink, \emph{Representations of Lattices by Sets}, Trans. Amer. Math. Soc. 64, 1948, 299--316.
\bibitem{bs} H. Bruhn, O. Schaudt, \emph{The Journey of the Union-Closed Sets Conjecture}, Graphs and Combinatorics 31(6), 2015, 2043--2074.
\bibitem{gg} G. Gr\"atzer: \emph{Lattice Theory: Foundation}, Birkh\"auser, 2011.
\bibitem{pb} B. Poonen, \emph{Union-Closed Families}, J. Combin. Theory (Series A) 59, 1992, 253--268.

\end{thebibliography}
\end{document}